\documentclass[12pt]{article}
\usepackage{amsmath,amsfonts,amssymb,amsthm}
\usepackage{eucal}

\newtheorem{Theorem}{Theorem}

\newtheorem{Lemma}[Theorem]{Lemma}
\newtheorem{Corollary}[Theorem]{Corollary}
\def\P{{\mathbb P}}
\def\E{{\mathbb E}}
\def\R{{\mathbb R}}

\def\I{{\mathbb I}}

\begin{document}

\section*{\centering\large\bf
           On subexponential tails for the maxima\\
           of negatively driven compound renewal
           and L\'evy processes}

\section*{\centering\normalsize\slshape\bfseries
Dmitry Korshunov\footnote{Lancaster University, UK.
E-mail: d.korshunov@lancaster.ac.uk}}

\begin{abstract}\noindent
We study subexponential tail asymptotics for the distribution
of the maximum $M_t:=\sup_{u\in[0,t]}X_u$ of
a process $X_t$ with negative drift for the entire range of $t>0$.
We consider compound renewal processes with linear drift and L\'evy processes.
For both we also formulate and prove the principle 
of a single big jump for their maxima.
The class of compound renewal processes with drift 
particularly includes Cram\'er--Lundberg renewal risk process.
\\[2mm]
{\it Keywords}: L\'evy process,
compound renewal process,
distribution tails,
heavy tails,
long-tailed distributions,
subexponential distributions,
random walk.

{\it AMS 2010 subject classification}: 60F10, 60G51, 60K05
\end{abstract}

\section{Introduction}\label{Introduction}

For a probability distribution $F$ on the real line,
let $F(x)=F(-\infty,x]$ denote the
distribution function and $\overline F(x)=F(x,\infty)=1-F(x)$ its tail. 
We say that $F$ is (right-) heavy-tailed distribution if
all its positive exponential moments are infinite, 
$\int_{\mathbb R} e^{sx}F(dx)=\infty$ for all $s>0$.
Otherwise we call $F$ (right-) light-tailed.

In the presence of heavy tails, the class ${\mathcal S}$
of subexponential distributions is of basic importance.
A distribution $F$ on ${\mathbb R}^+$ with unbounded support
is called {\it subexponential} if
$\overline{F*F}(x)\sim 2\overline F(x)$ as $x\to\infty$.
A distribution $F$ of $\xi$ on the whole real line is called subexponential
if the distribution $F^+$ of $\xi^+$ is so.

Any subexponential distribution is known 
(see, e.g., Foss et al. (2013, Lemma 3.2)) to be {\it long-tailed}, 
i.e., for any fixed $y$, $\overline F(x+y)\sim \overline F(x)$ as $x\to\infty$.

The class of subexponential distributions plays an important role 
in many applications, for instance, for waiting times in the $GI/G/1$ queue
and for ruin probabilities---see, e.g.,
Asmussen (2003, Ch. X.9); Asmussen and Albrecher (2010, Ch.\ X);
Embrechts et al. (1997, Sec. 1.4); %Embrechts and Veraverbeke (1982);
Rolski et al. (1998).

A distribution $F$ on ${\mathbb R}$
with right unbounded support and finite mean is called
{\it strong subexponential} ($F\in{\mathcal S}^*$) if
\begin{eqnarray*}
\int\limits_0^x \overline F(x-y)\overline F(y)dy &\sim&
2\overline F(x)\int\limits_0^\infty\overline F(y)dy
\quad\mbox{as }x\to\infty.
\end{eqnarray*}
It is known---see, e.g., Foss et al. (2013, Theorem 3.27)---that 
$F\in{\mathcal S}^*$ implies 
both $F\in{\mathcal S}$ and $F_I\in{\mathcal S}$ where $F_I$ is 
the {\it integrated tail distribution} defined by its tail,
\begin{eqnarray*}
\overline F_I(x) &:=& 
\min\Bigl(1,\int_x^\infty\overline F(y)dy\Bigr), \quad x>0.
\end{eqnarray*}

Let $Y$, $Y_1$, $Y_2$, \ldots\ be independent identically distributed
random variables with a negative expectation $b=\E Y<0$.
Consider a random walk $S_0=0$, $S_n=Y_1+\ldots+Y_n$ and its maximum
\begin{eqnarray*}
M^S_n &:=& \max_{0\le k\le n}\sum_{i=1}^k Y_i,
\end{eqnarray*}
hereinafter we follow the standard convention $\sum_{i=1}^0 f(i)=0$.

Since $b<0$, the family $M_n^S$, $n\ge 1$, is stochastically bounded.
Let $B$ be the distribution of $Y_1^+$ and $B_I$
be the integrated tail distribution of $Y_1^+$.
As well known for the overall maximum of the random walk,
\begin{eqnarray*}
M^S_\infty &=& \max_{n\ge 0}\sum_{i=1}^nY_i,
\end{eqnarray*}
the asymptotic relation
\begin{eqnarray}\label{max.global}
\P \{M^S_\infty>x\} &\sim& \overline B_I(x)/|b|
\quad\mbox{ as }x\to\infty
\end{eqnarray}
holds in the heavy-tailed case if and only if the integrated tail distribution $B_I$
is subexponential---see e.g. Theorem 5.12 in Foss et al. (2013). 
Also, if $B$ is strong subexponential, $B\in{\mathcal S}^*$, 
then the following tail result holds for finite time horizon maxima
\begin{eqnarray}\label{max.t}
\P\Bigl\{\max_{k\le n}\sum_{i=1}^k Y_i>x\Bigr\}
&\sim& \frac{1}{|b|}\int_x^{x+n|b|}\overline B(v)dv
\end{eqnarray}
as $x\to\infty$ uniformly for all $n\ge 1$---see Korshunov (2002)
or Foss et al. (2013, Theorem 5.3); uniformity for all $n\ge 1$ means that
\begin{eqnarray*}
\sup_{n\ge 1}\Biggl|\frac{\P\bigl\{\max_{k\le n}\sum_{i=1}^k Y_i>x\bigr\}}
{\frac{1}{|b|}\int_x^{x+n|b|}\overline B(v)dv}-1\Biggr|
&\to& 0\quad\mbox { as }x\to\infty.
\end{eqnarray*}

So the subexponential tail behaviour for the maxima of random walks 
is well understood while surprisingly much less is known for L\'evy processes.
In this contribution we particularly demonstrate in Section \ref{renewal} how 
results for random walks relate to those for the compound renewal process with 
linear drift in the presence of heavy-tails---see Theorem \ref{thm:renewal.linear};
in particular, we formulate and prove the principle of a single big jump
in Theorem \ref{single.big.jump.renewal}. 
Based on that we give in Section \ref{sec:Levy} a very general treatment 
of subexponential tail behaviour for L\'evy processes 
with negative drift---see Theorem \ref{thm:Levy.gen}. 
In Section \ref{Levy.random.time} we derive tail asymptotics for 
a L\'evy process stopped at random time and for its maximum 
within this time interval. An application to the Cram\'er--Lundberg
renewal risk model is given in Section \ref{sec:risk}. 
A discussion of results available in the literature may be found 
just after Theorems \ref{thm:renewal.linear} and \ref{thm:Levy.gen}.

\section{Asymptotics for compound renewal process}
\label{renewal}

Consider a {\it compound renewal process} $X_t$
which is defined as
\begin{eqnarray*}
X_t &=& \sum_{i=1}^{N_t}Y_i,
\end{eqnarray*}
where $N_t$ is a renewal process generated by jump epochs
$0=T_0<T_1<T_2<\ldots$, where $\tau_n:=T_n-T_{n-1}>0$
are independent identically distributed random
variables with finite mean $\E\tau=:1/\lambda$, and $Y_n$, $n\ge1$, 
are independent identically distributed jumps with finite mean $b$.
The $Y$'s are supposed to be independent of the process $N_t$.
Assume that the drift of the process is negative,
that is, $b<0$, so we have that the family of distributions of maxima
$$
M_t := \max_{u\in[0,t]}X_u
$$
is tight, 
$$
\sup_{t>0}\,\P\{M_t>x\} \le \P\{M_\infty>x\} \to 0 \quad\mbox{as }x\to\infty.
$$
We are interested in the tail behaviour of $M_t$. The overall maximum 
$M_\infty$ is simply the maximum of the associated random walk:
\begin{eqnarray*}
M_\infty &=& M^S_\infty\ =\ \max_{n\ge 0}\sum_{i=1}^nY_i,
\end{eqnarray*}
due to piecewise constant behaviour of the process $X_t$. 
Let $B$ be the distribution of $Y_1^+$ and $B_I$
be the integrated tail distribution of $Y_1^+$.
Then it follows from the result for the overall maximum 
of the associated random walk that
\begin{eqnarray}\label{max.global}
\P \{M_\infty>x\} &\sim& \overline B_I(x)/|b|
\quad\mbox{ as }x\to\infty
\end{eqnarray}
holds in the heavy-tailed case if and only if the integrated tail 
distribution $B_I$ is subexponential. 

The finite time horizon tail asymptotics for $M_t$ 
are slightly more complicated than that for the infinite time horizon
and are described in the following theorem.

\begin{Theorem}\label{thm:renewal}
Let $X_t$ be a compound renewal process with negative drift
$b/\E \tau<0$.
If the distribution $B$ of $Y_1^+$ is strong subexponential,
then, uniformly for all $t>0$,
\begin{eqnarray*}
\P \{M_t>x\} &\sim& \frac{1}{|b|}
\int_x^{x+|b|\E N_t}\overline B(v)dv
\quad\mbox{as }x\to\infty.
\end{eqnarray*}
In particular,
\begin{eqnarray*}
\P \{M_t>x\} &\sim& \frac{1}{|b|}
\int_x^{x+|b|\lambda t}\overline B(v)dv
\quad\mbox{as }x,\ t\to\infty.
\end{eqnarray*}
\end{Theorem}

For a {\it compound Poisson process} $X_t$ where
$N_t$ is a homogeneous Poisson process with intensity
of jumps $\lambda$, we have $\E N_t=t\lambda$, so the following corollary.

\begin{Corollary}\label{cor:Poisson}
Let $X_t$ be a compound Poisson process with negative drift
$\lambda b<0$.
If the distribution $B$ of $Y_1^+$ is strong subexponential,
then, uniformly for all $t>0$,
\begin{eqnarray*}
\P \{M_t>x\} &\sim& \frac{1}{|b|}
\int_x^{x+t|b|\lambda}\overline B(v)dv
\quad\mbox{as }x\to\infty.
\end{eqnarray*}
\end{Corollary}

Theorem \ref{thm:renewal} follows from a more general
result stated next. It concerns a 
{\it compound renewal process with linear drift}, that is,
\begin{eqnarray*}
X_t &=& \sum_{i=1}^{N_t}Y_i+ct,
\end{eqnarray*}
where $N_t$ and the $Y$'s are as above while $c$ is 
some real constant. 
Notice that the random variables $Y_i+c\tau_i$ depend on $N_t$.
We assume that the drift of the process is negative, that is, 
$c+b\lambda<0$ which implies that the family of distributions of 
maxima $M_t := \max_{u\in[0,t]}X_t$ is tight.

\begin{Theorem}\label{thm:renewal.linear}
Let $X_t$ be a compound renewal process with linear drift such that
$a:=c/\lambda+b<0$. Let the distribution $B$ of $Y_1^+$ be
strong subexponential and one of the following conditions hold:

{\rm(i)} $c\le 0$;

{\rm(ii)} $c>0$ and $\P\{c\tau>x\}=o(\overline B(x))$ as $x\to\infty$.\\
Then, uniformly for all $t>0$,
\begin{eqnarray}\label{ren.drift.asy}
\P \{M_t>x\} &\sim& \frac{1}{|a|}
\int_x^{x+|a|\E N_t}\overline B(v)dv
\quad\mbox{as }x\to\infty.
\end{eqnarray}
In particular,
\begin{eqnarray*}
\P \{M_t>x\} &\sim& \frac{1}{|a|}
\int_x^{x+|a|\lambda t}\overline B(v)dv
\quad\mbox{as }x,\ t\to\infty.
\end{eqnarray*}
\end{Theorem}

A particular case of this result was proven in Foss et al. (2013)
by alternative techniques
in the context of {\it Cram\'er--Lundberg collective risk model} 
where $N_t$ is a Poisson process and $c<0$---see Theorem 5.21 there.
In the book by Borovkovs (2008, Ch. 16) the tail behavior of $M_t$ 
is only described for $t\to\infty$ 
and for regularly varying distribution of $Y_1$.

If the linear drift coefficient is positive, that is $c>0$,
and if the condition $\P\{c\tau>x\}=o(\overline B(x))$ fails,
then the tail asymptotics of the distribution of $M_t$
may be more complicated. In particular, then
$\P\{M_t>x\}\ge\P \{\tau_1>x/c\}$, so the tail of $M_t$ may be heavier
than the integrated tail of $B$ if the tail of $\tau$ is so.
We do not concern tail asymptotics for $M_t$ in the general case when $c>0$; 
we only present the following result on the overall maximum $M_\infty$:

\begin{Theorem}\label{thm:renewal.linear.infty}
Let $X_t$ be a compound renewal process such that $a:=c/\lambda+b<0$
and the integrated tail distribution $F_I$ of $c\tau_1+Y_1^+$ 
is subexponential. Then
\begin{eqnarray*}
\P \{M_\infty>x\} &\sim& \frac{1}{|a|}
\int_x^\infty\overline F(v)dv
\quad\mbox{ as }x\to\infty.
\end{eqnarray*}
\end{Theorem}

Notice that the distribution of $c\tau_1+Y_1$ is strong subexponential
in the case $c\le 0$ if and only if the distribution of $Y_1$ 
is strong subexponential.

\begin{proof}[Proof of Theorem \ref{thm:renewal.linear}.]
First let us prove that, for any fixed $t_0$, 
\eqref{ren.drift.asy} holds uniformly for all $t\le t_0$. 
Indeed, for all $t\le t_0$,
\begin{eqnarray*}
\P\Bigl\{\sum_{n=1}^{N_t}Y_n>x+|c|t_0\Bigr\}\ \le\
\P\{M_t>x\}
\ \le\ \P\Bigl\{\sum_{n=1}^{N_t}Y_n^+>x-|c|t_0\Bigr\}.
\end{eqnarray*}
Since the $Y$'s are strong subexponential, they are particularly subexponential.
For the renewal process $N_t$, there exists a $\delta>0$ such that
\begin{eqnarray*}
\sup_{t\le t_0}\E(1+\delta)^{N_{t_0}} &<& \infty.
\end{eqnarray*}
Together with independence of the $Y$'s and $N_t$,
it allows to apply Kesten's bound---see e.g. 
Foss et al. (2013, Theorem 3.34)---and to conclude the following uniform 
in $t\le t_0$ analogue of the tail result for randomly stopped 
sums---see Foss et al. (2013, Theorem 3.37):
\begin{eqnarray*}
\P\Bigl\{\sum_{n=1}^{N_t}Y_n>x\Bigr\}  &\sim& \E N_t\P\{Y_1>x\}
\quad\mbox{as }x\to\infty\mbox{ uniformly for all }t\le t_0.
\end{eqnarray*}
The same arguments work for the $Y^+$'s. Therefore, 
\begin{eqnarray*}
(1+o(1))\E N_t\P\{Y_1>x+|c|t_0\}\ \le\ \P\{M_t>x\}
\ \le\ (1+o(1))\E N_t\P\{Y_1>x-|c|t_0\}
\end{eqnarray*}
as $x\to\infty$ uniformly for all $t\le t_0$.
Subexponentiality of $Y$'s implies $B$ is long-tailed, so hence
\begin{eqnarray*}
\P\{M_t>x\} &\sim& \E N_t\P\{Y_1>x\}
\quad\mbox{as }x\to\infty\mbox{ uniformly for all }t\le t_0,
\end{eqnarray*}
which is equivalent to the fact that
\eqref{ren.drift.asy} holds uniformly for all $t\le t_0$ because
\begin{eqnarray*}
\frac{1}{|b|} \int_x^{x+|b|\E N_t}\overline B(v)dv
&\sim& \E N_t\overline B(x)
\quad\mbox{as }x\to\infty\mbox{ uniformly for all }t\le t_0,
\end{eqnarray*}
again by long-tailedness of $B$.

Therefore there exists an increasing function $h(x)\to\infty$
such that \eqref{ren.drift.asy} holds uniformly for all $t\le h(x)$.

Then it remains to prove \eqref{ren.drift.asy} 
for the range $t>h(x)$ where the above arguments clearly do not help.
Instead, we proceed with a standard technique of getting
the lower and upper bounds for the tail of $M_t$ which
are asymptotically equivalent. 
For the lower bound, fix an $\varepsilon>0$. 
By the strong law of large numbers, there exists an $A$ such that
\begin{eqnarray}\label{SLLN}
\P \{|T_n-n\E \tau|<n\varepsilon+A \mbox{ for all }n\ge 1\}
&\ge& 1-\varepsilon.
\end{eqnarray}
Notice that
\begin{eqnarray*}
\P \{M_t>x\} &\ge& \P \Bigl\{\sum_{i=1}^n Y_i+cT_n>x
\mbox{ for some }n\le N_t\Bigr\}.
\end{eqnarray*}
On the event \eqref{SLLN}, if $t\ge n(\E \tau+\varepsilon)+A$
(equivalently, $n\le \bigl[\frac{t-A}{\E \tau+\varepsilon}\bigr]=:n(t)$)
then $T_n\le t$ and hence $n\le N_t$.
Since the jumps $Y$'s do not depend on the renewal process $N_s$, 
we obtain the inequality
\begin{eqnarray*}
\P \{M_t>x\} &\ge& (1-\varepsilon)\P \Bigl\{\sum_{i=1}^n Y_i
+c(n(\E \tau+\varepsilon)+A)>x
\mbox{ for some }n\le n(t)\Bigr\}
\end{eqnarray*}
for $c\le 0$ and the inequality
\begin{eqnarray*}
\P \{M_t>x\} &\ge& (1-\varepsilon)\P \Bigl\{\sum_{i=1}^n Y_i
+c(n(\E \tau-\varepsilon)-A)>x
\mbox{ for some }n\le n(t)\Bigr\}
\end{eqnarray*}
for $c>0$. Thus, in both cases,
\begin{eqnarray*}
\P \{M_t>x\}
&\ge& (1-\varepsilon)\P \Bigl\{\max_{0\le n\le n(t)}
\sum_{i=1}^n (Y_i+c\E \tau-|c|\varepsilon)>x+|c|A\Bigr\}.
\end{eqnarray*}
Applying the equivalence \eqref{max.t}
we obtain the following lower bound:
\begin{eqnarray*}
\P \{M_t>x\}
&\ge& \frac{1-\varepsilon+o(1)}{|b+c\E \tau-|c|\varepsilon|}
\int_{x+|c|A}^{x+|c|A+n(t)|b+c\E \tau-|c|\varepsilon|}
\overline B(v)dv\\
&\sim& \frac{1-\varepsilon}{|a-|c|\varepsilon|}
\int_0^{t\frac{|a-|c|\varepsilon|}{\E\tau+\varepsilon}}
\overline B(x+v)dv\quad\mbox{ as }x,\ t\to\infty, 
\end{eqnarray*}
because $B$ is a long-tailed distribution. 
Taking into account that, for every $\gamma>0$,
\begin{eqnarray*}
\int_0^{\gamma t}\overline B(x+u)du &\ge& 
\min(1,\gamma) \int_0^t\overline B(x+u)du, 
\end{eqnarray*}
we conclude that
\begin{eqnarray*}
\P \{M_t>x\}
&\ge& \frac{1-\varepsilon+o(1)}{|a-|c|\varepsilon|}
\min\Bigl(1,\frac{|a-|c|\varepsilon|}{\E\tau+\varepsilon}
\frac{\E\tau}{|a|}\Bigr)
\int_x^{x+|a|t/\E\tau}\overline B(v)dv
\end{eqnarray*}
as $x$, $t\to\infty$.
Letting $\varepsilon\downarrow 0$ completes the proof of the lower bound
\begin{eqnarray*}
\P \{M_t>x\}
&\ge& \frac{1+o(1)}{|a|}\int_x^{x+|a|t/\E\tau}\overline B(v)dv
\quad\mbox{ as }x,\ t\to\infty.
\end{eqnarray*}

Now let us turn to the upper bound for $\P\{M_t>x\}$.
First consider the case $c\le 0$ when the trajectory of $X_t$ linearly 
drops down between jumps and  the maximum may be only attained
at a jump epoch,
\begin{eqnarray*}
M_t &=& \max_{0\le n\le N_t}\ \sum_{i=1}^n (Y_i+c\tau_i).
\end{eqnarray*}
Therefore, for any $\varepsilon>0$,
\begin{eqnarray}\label{MXt.upper}
\P\{M_t>x\} &\le& \P\Bigl\{\max_{0\le n\le(1+\varepsilon)\E N_t}\ 
\sum_{i=1}^n (Y_i+c\tau_i)>x\Bigr\}\nonumber\\
&&\hspace{30mm}+\P\Bigl\{\sum_{i=1}^{N_t} Y_i^+>x,\ N_t>(1+\varepsilon)\E N_t\Bigr\}.
\end{eqnarray}
The distribution of $Y$ is strong subexponential and $c<0$,
so $Y+c\tau$ is strong subexponential too and
$$
\P\{Y+c\tau>x\}\sim\P\{Y>x\}=\overline B(x)\quad\mbox{as }x\to\infty.
$$ 
Thus, by \eqref{max.t},
\begin{eqnarray}\label{gehx.1}
\P\Bigl\{\max_{0\le n\le(1+\varepsilon)\E N_t}\ 
\sum_{i=1}^n (Y_i+c\tau_i)>x\Bigr\}
&\sim& \frac{1}{|a|}\int_x^{x+|a|(1+\varepsilon)\E N_t} \overline B(y)dy\nonumber\\
&\le& \frac{1+\varepsilon}{|a|}\int_x^{x+|a|\E N_t} \overline B(y)dy,
\end{eqnarray}
because $\overline B(y)$ is decreasing. Further,
\begin{eqnarray}\label{renewal.0}
\lefteqn{\P\Bigl\{\sum_{i=1}^{N_t} Y_i^+>x,\ N_t>(1+\varepsilon)\E N_t\Bigr\}}
\nonumber\\
&&\hspace{12mm}= \sum_{k=1}^\infty\P\Bigl\{
\sum_{i=1}^{N_t} Y_i^+>x,\
(1+k\varepsilon)\E N_t<N_t\le(1+(k+1)\varepsilon)\E N_t\Bigl\}\nonumber\\
&&\hspace{32mm}\le \sum_{k=1}^\infty\P\Bigl\{
\sum_{i=1}^{(1+(k+1)\varepsilon)\E N_t} Y_i^+>x\Bigr\}
\P\{N_t>(1+k\varepsilon)\E N_t\},
\end{eqnarray}
owing to independence of $Y$'s and $N_t$.
Denote $K:=[(1+k\varepsilon)\E N_t]$. Then
\begin{eqnarray*}
\P\{N_t>(1+k\varepsilon)\E N_t\} &=& \P\{T_K \le t\}\\
&=& \P\bigl\{K\E\tau(1-\varepsilon/2)-T_K \ge K\E\tau(1-\varepsilon/2)-t\bigr\}\\
&\le& \P\bigl\{K\E\tau(1-\varepsilon/2)-T_K \ge 0\bigr\}
\end{eqnarray*}
for sufficiently large $t$ and $\varepsilon\in(0,1)$ because, as $t\to\infty$, 
\begin{eqnarray*}
K\E\tau(1-\varepsilon/2)-t &\sim& 
t((1+k\varepsilon)(1-\varepsilon/2)-1)\\
&\ge& t(\varepsilon/2-\varepsilon^2/2)\ >\ 0.
\end{eqnarray*}
Since the random variable $\E\tau(1-\varepsilon/2)-\tau$ 
has negative expectation $-\varepsilon\E\tau/2$ and is bounded from above
by $\E\tau(1-\varepsilon/2)$, there exists 
a $\beta=\beta(\varepsilon)>0$ such that 
\begin{eqnarray*}
\E e^{\beta(\E\tau(1-\varepsilon/2)-\tau)} &=& 1-\delta<1.
\end{eqnarray*} 
Hence, by exponential Chebyshev's inequality, 
\begin{eqnarray*}
\P\bigl\{K\E\tau(1-\varepsilon/2)-T_K \ge 0\bigr\}
&\le& (1-\delta)^K
\end{eqnarray*}
for all $k\ge 1$ and sufficiently large $t$, so 
\begin{eqnarray}\label{renewal.1}
\P\{N_t>(1+k\varepsilon)\E N_t\} 
&\le& (1-\delta)^{[(1+k\varepsilon)\E N_t]}.
\end{eqnarray}
By Kesten's bound---see e.g. Foss et al. (2013, Theorem 3.34)---there 
is an $A<\infty$ such that
\begin{eqnarray*}
\P\Bigl\{\sum_{i=1}^{(1+(k+1)\varepsilon)\E N_t} Y_i^+>x\Bigr\}
&\le& A(1+\delta/8)^{(1+(k+1)\varepsilon)\E N_t}\P\{Y>x\}
\end{eqnarray*}
for all $x>0$, $k\ge 1$ and $t>0$. For $k\ge 1$ and sufficiently large $t$,
\begin{eqnarray*}
(1+(k+1)\varepsilon)\E N_t &\le& 2[(1+k\varepsilon)\E N_t],
\end{eqnarray*}
thus
\begin{eqnarray}\label{renewal.2}
\P\Bigl\{\sum_{i=1}^{(1+(k+1)\varepsilon)\E N_t} Y_i^+>x\Bigr\}
&\le& A(1+\delta/8)^{2[(1+k\varepsilon)\E N_t]}\P\{Y>x\}\nonumber\\
&\le& A(1+\delta/2)^{[(1+k\varepsilon)\E N_t]}\P\{Y>x\}.
\end{eqnarray}
Substituting \eqref{renewal.1} and \eqref{renewal.2} into \eqref{renewal.0}
and taking into account that $(1-\delta)(1+\delta/2)\le 1-\delta/2$,
we obtain, for all sufficiently large $t$,
\begin{eqnarray*}
\P\Bigl\{\sum_{i=1}^{N_t} Y_i^+>x,\ N_t>(1+\varepsilon)\E N_t\Bigr\}
&\le& A\P\{Y>x\}\sum_{k=1}^\infty
(1-\delta/2)^{[(1+k\varepsilon)\E N_t]}.
\end{eqnarray*}
The sum on the right goes to zero as $t\to\infty$. 
Therefore, for any fixed $\varepsilon>0$,
\begin{eqnarray*}
\P\Bigl\{\sum_{i=1}^{N_t} Y_i^+>x,\ N_t>(1+\varepsilon)\E N_t\Bigr\}
&=& o(\P\{Y>x\})
\end{eqnarray*}
as $t\to\infty$ uniformly for all $x>0$.
Combining this bound with \eqref{gehx.1} we get
\begin{eqnarray*}
\P\{M_t>x\} &\le& \frac{1+\varepsilon+o(1)}{|a|}
\int_x^{x+|a|\E N_t} \overline B(y)dy
\end{eqnarray*}
as $x\to\infty$ uniformly for $t\ge h(x)$.
Letting $\varepsilon\downarrow0$, we conclude 
\begin{eqnarray*}
\P\{M_t>x\} &\le& \frac{1+o(1)}{|a|}
\int_x^{x+|a|\E N_t}\overline B(v)dv
\end{eqnarray*}
as $x\to\infty$ uniformly for $t\ge h(x)$.
This proves Theorem \ref{thm:renewal.linear} in the case $c\le 0$.

Now consider the case $c>0$ when the trajectory of $X_t$ linearly 
grows between jumps and the maximum may be only attained 
just prior to a jump epoch or at time $t$, so hence
\begin{eqnarray}\label{MXt.tau.hatMt}
M_t &\le& c\tau_1+\max_{0\le n\le N_t}
\Bigl(\sum_{i=1}^n Y_i+c((T_{n+1}-\tau_1)\wedge t)\Bigr)\nonumber\\
&=:& c\tau_1+\widehat M_t,
\end{eqnarray}
where $\tau_1$ and $\widehat M_t$ are independent.
Similar to the case $c\le 0$, for any $\varepsilon>0$,
\begin{eqnarray}\label{hat.M.above}
\P\{\widehat M_t>x\} &\le& \P\Bigl\{\max_{0\le n\le(1+\varepsilon)\E N_t}\ 
\sum_{i=1}^n Y_i+ct>x\Bigr\}\nonumber\\
&&\hspace{20mm}+\P\Bigl\{\sum_{i=1}^{N_t} Y_i^++ct>x,\ N_t>(1+\varepsilon)\E N_t\Bigr\}.
\end{eqnarray}
The distribution of $Y$ is strong subexponential and 
the tail of $c\tau$ is of order $o(\overline B(x))$,
so $Y+c\tau$ is strong subexponential too and
$$
\P\{Y+c\tau>x\}\sim\P\{Y>x\}=\overline B(x)\quad\mbox{as }x\to\infty.
$$ 
Thus, by \eqref{max.t}, we get that the first term on the right hand side
of \eqref{hat.M.above} possesses the upper bound \eqref{gehx.1}.
The second term on the right hand side of \eqref{hat.M.above}
may be bounded from above as follows. Take $c_1$ so large that
$c_1\E N_t\ge t$ for all $t>1$. Then
\begin{eqnarray*}
\lefteqn{\P\Bigl\{\sum_{i=1}^{N_t} Y_i^++ct>x,\ N_t>(1+\varepsilon)\E N_t\Bigr\}}\\
&&\hspace{12mm} =\ \P\Bigl\{\sum_{i=1}^{N_t} (Y_i^++c_1)+ct-c_1N_t>x,\ N_t>(1+\varepsilon)\E N_t\Bigr\}\\
&&\hspace{28mm} \le\ \P\Bigl\{\sum_{i=1}^{N_t} (Y_i^++c_1)>x,\ N_t>(1+\varepsilon)\E N_t\Bigr\},
\end{eqnarray*}
which possesses the same upper bound as the second term on the
right hand side of \eqref{MXt.upper}. Altogether it implies that
\begin{eqnarray*}
\P\{\widehat M_t>x\} &\le& 
\frac{1+o(1)}{|a|}\int_x^{x+|a|\E N_t} \overline B(y)dy
\quad\mbox{as }x\to\infty.
\end{eqnarray*}
Since $c\tau_1$ and $\widehat M_t$ in \eqref{MXt.tau.hatMt} are independent,
\begin{eqnarray*}
\P\{M_t>x\} &\le& 
\P\{c\tau_1>x\}+\int_0^x \P\{c\tau_1\in du\} \P\{\widehat M_t>x-u\}
\end{eqnarray*}
which allows to carry out standard calculations for subexponential 
distributions based on the condition $\P\{c\tau_1>x\}=o(\overline B(x))$
and the upper bound for $\widehat M_t$ and to conclude the upper bound
\begin{eqnarray*}
\P\{M_t>x\} &\le& 
\frac{1+o(1)}{|a|}\int_x^{x+|a|\E N_t} \overline B(y)dy
\quad\mbox{as }x\to\infty.
\end{eqnarray*}
which completes the proof in the case $c>0$.
The proof of Theorem \ref{thm:renewal.linear} is complete.
\end{proof}

\begin{proof}[Proof of Theorem \ref{thm:renewal.linear.infty}.]
We need only to consider the case $c>0$. Then the lower bound
for the tail of $M_\infty$ follows from the inequality
\begin{eqnarray*}
M_\infty &\ge& \sup_{k\ge 0}\sum_{i=1}^k (Y_i+c\tau_{i+1})=:\zeta
\end{eqnarray*}
and from the result \eqref{max.global} for maxima of sums. 
The upper bound follows from the equality
\begin{eqnarray*}
M_\infty &=& c\tau_1+\zeta
\end{eqnarray*}
and from the observation that 
\begin{eqnarray*}
\P\{c\tau_1>x\} &=& O(\overline F(x))
=o(\overline F_I(x))\quad\mbox{as }x\to\infty
\end{eqnarray*}
which allows to apply \cite[Corollary 3.18]{FKZ}. The proof is complete.
\end{proof}

We conclude this section with the following theorem
which is nothing other than the {\it principle of a single big
jump} for the maximum $M_t$.
For any $A>0$ and $\varepsilon>0$ consider events
\begin{eqnarray}\label{def.of.Dk}
D_k &:=& \bigl\{|X_s-a\lambda s|\le \varepsilon s+A
\mbox{ for all }s<T_k,\ Y_k>x+|a|\lambda T_k\bigr\}
\end{eqnarray}
which, for large $x$, roughly speaking means that
up to time $T_k$ the process $X_s$ drifts down with rate $a$
according to the strong law of large numbers and
then makes a big jump up at time $T_k$ of size $x$
plus value that compensates the negative drift up to this time. 
As stated in the next theorem, the union of these events describes 
the most probable way by which large deviations of $M_t$ 
can occur---it is very different from what is observed
if $Y$'s possess some positive exponential moment finite.
It is an analogue for discrete time process of the
principle of a single big jump for the maximum of a random walk
with negative drift, see Theorem 5.4 in Foss et al. (2013).

\begin{Theorem}\label{single.big.jump.renewal}
In conditions of Theorem \ref{thm:renewal.linear}, 
for any fixed $\varepsilon>0$,
\begin{eqnarray*}
\lim_{A\to\infty}\lim_{t,x\to\infty} \P\{\cup_{k=1}^{N_t} D_k|M_t>x\}
&\ge& \frac{|a|}{|a|+2\varepsilon/\lambda}.
\end{eqnarray*}
\end{Theorem}

\begin{proof}
Choose $\gamma>0$ so small that 
$(|a|\lambda+\varepsilon)(1/\lambda+\gamma\varepsilon)<|a|+2\varepsilon/\lambda$
for all $\varepsilon\in(0,1)$.
Then, since, for $k$ such that $k(\E\tau+\gamma\varepsilon)+A\le t$, 
each of the events
\begin{eqnarray*}
\widetilde D_k &:=& \bigl\{|X_s-a\lambda s|\le \varepsilon s+A\mbox{ for all }s<T_k,
\ T_j\le j(\E\tau+\varepsilon\gamma)+A\mbox{ for all }j\le k,\\ 
&&\hspace{50mm} M_{T_k-0}\le x,\ 
Y_k>x+A+T_k(|a|\lambda+\varepsilon)\bigr\}
\end{eqnarray*}
is contained in $T_k\le t$ and in $D_k$ and implies 
that $M_{T_k}>x$ because on the event $\widetilde D_k$ we have
\begin{eqnarray*}
X_{T_k} &=& X_{T_k-0}+Y_k\\
&>& (a\lambda -\varepsilon)T_k-A+x+A+T_k(|a|\lambda+\varepsilon)\ =\ x,
\end{eqnarray*}
so that $M_t>x$. 
Then, for $N:=[\frac{t-A}{\E\tau+\gamma\varepsilon}]$, 
we consequently have that
\begin{eqnarray}\label{uni.B.M.ge}
\P\{\cup_{k=1}^{N_t} D_k|M_t>x\} &\ge& 
\P\{\cup_{k=1}^{N}\widetilde D_k|M_t>x\}
\ =\ \frac{\P\{\cup_{k=1}^{N}\widetilde D_k\}}
{\P\{M_t>x\}}.
\end{eqnarray}
The events $\widetilde D_k$ are disjoint, hence
\begin{eqnarray*}
\P\{\cup_{k=1}^{N}\widetilde D_k\}
&=& \sum_{k=1}^{N}\P\{\widetilde D_k\}.
\end{eqnarray*}
It follows from the strong law of large numbers applied
to both $X_s$ and $N_s$ that, for any fixed $\delta>0$, 
there exists an $A$ such that, for all $x>A$,
\begin{eqnarray*}
\P\{\cup_{k=1}^{N}\widetilde D_k\}
&\ge& (1-\delta/4)\sum_{k=1}^{N} 
\P\{Y_k>x+A+T_k(|a|\lambda+\varepsilon)\mid T_k\le k(\E\tau+\varepsilon\gamma)+A\}\\
&\ge& (1-\delta/4)\sum_{k=1}^{N} 
\P\{Y_k>x+(1+|a|\lambda+\varepsilon)A+k(|a|\lambda+\varepsilon)(\E\tau+\varepsilon\gamma)\}\\
&\ge& (1-\delta/4)\sum_{k=1}^{N} 
\P\{Y_k>x+(1+|a|\lambda+\varepsilon)A+k(|a|+2\varepsilon/\lambda)\},
\end{eqnarray*}
by the choice of the $\gamma>0$. Since the distribution $B$ is long-tailed,
\begin{eqnarray*}
\P\{\cup_{k=1}^{N}\widetilde D_k\}
&\ge& (1-\delta/2)\sum_{k=0}^{N-1} 
\P\{Y_k>x+k(|a|+2\varepsilon/\lambda)\}
\end{eqnarray*}
for all sufficiently large $x$. Hence
\begin{eqnarray*}
\P\{\cup_{k=0}^{N-1}\widetilde D_k\}
&\ge& \frac{1-\delta/2}{|a|+2\varepsilon/\lambda}
\int_x^{x+N(|a|+2\varepsilon/\lambda)}\overline B(y)dy,
\end{eqnarray*}
because $\overline B(y)$ decreases. 
Take also into account that, for some $c_1<\infty$,
\begin{eqnarray*}
N(|a|+2\varepsilon/\lambda) &\ge& 
t\frac{|a|+2\varepsilon/\lambda}{\E\tau+\gamma\varepsilon}-c_1
\ \ge\ t(|a|\lambda+\varepsilon)-c_1,
\end{eqnarray*}
owing the choice of $\gamma>0$, so
\begin{eqnarray*}
N(|a|+2\varepsilon/\lambda) &\ge& t|a|\lambda
\quad\mbox{for all sufficiently large }t.
\end{eqnarray*}
Then we deduce
\begin{eqnarray*}
\P\{\cup_{k=1}^{N}\widetilde D_k\}
&\ge& \frac{1-\delta/2}{|a|+2\varepsilon/\lambda}
\int_x^{x+t|a|\lambda}\overline B(y)dy.
\end{eqnarray*}
Substituting this estimate and the asymptotics
for $M_t$ into \eqref{uni.B.M.ge} we deduce that
\begin{eqnarray*}
\lim_{t, x\to\infty} \P\{\cup_{k=1}^{N_t} D_k|M_t>x\}
&\ge& \frac{(1-\delta)|a|}{|a|+2\varepsilon/\lambda}.
\end{eqnarray*}
Now we can make $\delta>0$ as small as we please
by choosing a sufficiently large $A$. This completes the proof.
\end{proof}

\section{Asymptotics for L\'evy process}
\label{sec:Levy}

Let $X_t$ be a c\`adl\`ag stochastic process in $\R$
which means  that its paths are right continuous with left limits
everywhere, with probability $1$. Then, for every $t$,
the supremum
$$
M_t:=\sup_{u\in[0,t]}X_u
$$
is finite a.s.
In this section we study tail behaviour of
the distribution of $M_t$ for a {\it L\'evy process} $X_t$
starting at the origin, that is, for a stochastic process
with stationary independent increments, where stationary
means that, for $s<t$, the probability distribution of
$X_t-X_s$ depends only on $t-s$ and where independent
increments means that that difference $X_t-X_s$ is
independent of the corresponding difference on any
interval not overlapping with $[s,t]$, and similarly
for any finite number of mutually non-overlapping intervals.
Our main result for L\'evy processes is the following theorem.

\begin{Theorem}\label{thm:Levy.gen}
Assume the finite mean and negative drift, $a:=\E X_1<0$.
If the integrated tail distribution $F_I$ of $X_1$ is subexponential, then
\begin{eqnarray*}
\P \bigl\{\max_{u>0}X_u>x\bigr\}
&\sim& \frac{1}{|a|} \int_x^\infty \overline F(v)dv
\quad\mbox{ as }x\to\infty.
\end{eqnarray*}
If the distribution $F$ of $X_1$ is strong subexponential,
then, uniformly for all $t>0$,
\begin{eqnarray*}
\P \bigl\{\max_{u\in[0,t]}X_u>x\bigr\}
&\sim& \frac{1}{|a|} \int_x^{x+t|a|}\overline F(v)dv
\quad\mbox{ as }x\to\infty.
\end{eqnarray*}
\end{Theorem}

It has been suggested by Asmussen and Kl\"uppelberg (1996) 
and by Asmussen (1998) to follow a discrete skeleton argument 
in order to prove this asymptotics for $t=\infty$ 
when the tail of the L\'evy measure is subexponential; 
notice that this approach requires additional considerations 
which take into account fluctuations of L\'evy processes within time slots;
see the remark after Theorem \ref{thm:Levy.1}. 

In Braverman et al. (2002) tail asymptotics are presented for
some subclass of subadditive functionals of L\'evy process
with regularly varying at infinity L\'evy measure.
The overall supremum is a particular example considered in that article.

In Kl\"uppelberg et al. (2004, Theorem 6.2), tail asymptotics for the overall
supremum of negatively driven L\'evy process are derived via
direct approach based on ladder properties of the L\'evy process.

In Doney et al. (2016) the passage time problem is considered
for L\'evy processes, emphasising heavy tailed cases;
local and functional versions of limit distributions are derived 
for the passage time itself, as well as for the position 
of the process just prior to passage, and the overshoot of a high level
which is an extension for L\'evy processes of corresponding results 
for random walks, see e.g. Foss et al. (2013, Theorem 5.24).

In Foss et al. (2007, Theorem 3.1), Markov modulated L\'evy process 
is studied and again the tail asymptotics for the overall supremum 
were proven, via reduction to Markov modulated random walk.

In the book by Borovkovs (2008, Ch. 15) some partial results
on $\max_{u\in[0,t]}X_u$ are formulated 
(see, for example, Theorems 15.2.2(vi) and 15.3.12 there) 
under some specific conditions on the distribution of $X_1$; 
the supporting arguments provided may be hardly considered 
as clear and comprehensive. 
For example, on page 525 the authors justify transition from
integer $t$ to non-integer $t$ by convergence in probability 
$X_u\to 0$ as $u\to\infty$ which is clearly insufficient.
Also notice that it was not proven there that the corresponding
asymptotics hold uniformly for all $t>0$.

Related results on sample-path large deviations of scaled 
L\'evy processes $X(nt)/n$ with regularly varying L\'evy measure 
are proven by Rhee et al. (2016).

The following result is due to Willekens \cite{W};
it was proven via natural elementary rather short arguments. 

\begin{Theorem}\label{thm:Levy.1}
Let $X_t$ be a L\'evy process. For any fixed $t>0$,
the following assertions are equivalent:

{\rm(i)} the distribution of $X_t$ is long-tailed;

{\rm(ii)} the distribution of $M_t$ is long-tailed.

Each of {\rm(i)} and {\rm(ii)} implies
\begin{eqnarray}\label{fixed.t}
\P \{M_t>x\} &\sim& \P \{X_t>x\} \quad\mbox{ as }x\to\infty.
\end{eqnarray}
\end{Theorem}

Notice that Theorem \ref{thm:Levy.1} together with
Theorem 1 for regenerative processes from Palmowski and Zwart 
(2007)---or with Theorem 3.3 from Asmussen et al. (1999)---provides 
a correct version of skeleton approach for proving subexponential asymptotics 
for the overall supremum $M_\infty$ under negative drift assumption.

In our proof of Theorem \ref{thm:Levy.gen} we need the 
following lemma which may be of independent interest.

\begin{Lemma}\label{long.tailed.cri}
Let $G$ and $B$ be two distributions on $\R$ and let $G$ be light-tailed,
that is, there exist $\lambda>0$ and $c<\infty$ such that 
$\overline G(x)\le ce^{-\lambda x}$ for all $x$. 
Denote $F:=G*B$.

{\rm (i)} If $B$ is long-tailed then 
$\overline F(x)\sim\overline B(x)$ as $x\to\infty$;
in particular, $F$ is long-tailed too.

{\rm (ii)} If $F$ is long-tailed then $B$ is long-tailed too.
\end{Lemma}

Similar proposition was proven for subexponential distributions in
Embrechts et al. (1979, Proposition 1); our proof is similar.

\begin{proof}
(i) Assume that $B$ is long-tailed. Then there exists an increasing 
function $h(x)\to\infty$ such that (see Foss et al. (2011, Lemma 2.19)
\begin{eqnarray}\label{B.long.tailed-}
\overline B(x-h(x)) &\sim& \overline B(x) \quad\mbox{ as }x\to\infty.
\end{eqnarray}
Consider the following decomposition:
\begin{eqnarray*}
\frac{\overline F(x)}{\overline B(x)} 
&=& \int_{-\infty}^{h(x)} \frac{\overline B(x-y)}{\overline B(x)} G(dy)
+\int_{h(x)}^\infty \frac{\overline B(x-y)}{\overline B(x)} G(dy)\\
&=:& I_1(x)+I_2(x).
\end{eqnarray*}
Since
\begin{eqnarray*}
\frac{\overline B(x-y)}{\overline B(x)} 
&\le& \frac{\overline B(x-h(x))}{\overline B(x)}
\end{eqnarray*}
for all $y\le h(x)$, it follows from \eqref{B.long.tailed-} 
that the integrand in $I_1(x)$ possesses an integrable majorant.
Moreover, for every $y$, $\frac{\overline B(x-y)}{\overline B(x)}\to 1$
as $x\to\infty$. Hence, by the dominated convergence theorem,
\begin{eqnarray}\label{I1}
I_1(x) &\to& 1\quad\mbox{as }x\to\infty.
\end{eqnarray}
Further, since the distribution $B$ is long-tailed, for any $\varepsilon>0$
there exists $x(\varepsilon)$ such that
\begin{eqnarray*}
\overline B(x-1) &\le& \overline B(x)e^\varepsilon
\quad\mbox{for all }x\ge x(\varepsilon).
\end{eqnarray*}
Hence, there exists $c(\varepsilon)<\infty$ such that
\begin{eqnarray*}
\overline B(x-y) &\le& c(\varepsilon)\overline B(x)e^{\varepsilon y}
\quad\mbox{for all }x\ge x(\varepsilon),\ y>0.
\end{eqnarray*}
Take $\varepsilon<\lambda$. Then
\begin{eqnarray*}
I_2(x) &\le& c(\varepsilon)\int_{h(x)}^\infty e^{\varepsilon y}G(dy)
\to 0\quad\mbox{ as }x\to\infty,
\end{eqnarray*}
because $\overline G(x)=O(e^{-\lambda x})$ and $h(x)\to\infty$ as $x\to\infty$.
Together with \eqref{I1} it implies the relation 
$\overline F(x)\sim\overline B(x)$ as $x\to\infty$.

(ii) Assume that $F$ is long-tailed. Let us then prove that 
$\overline B(x)\sim\overline F(x)$ which implies long-tailedness of $B$.
Since $F$ is long-tailed, there exists a function $h(x)\to\infty$ 
such that $\overline F(x-h(x))\sim\overline F(x)$ as $x\to\infty$. 

For every $x$ and $h\in\R$ the following inequality holds:
$$
\overline F(x-h)=\overline{G*B}(x-h)\ge\overline G(-h)\overline B(x).
$$
If we choose $h_0$ satisfying $\overline G(-h_0)\ge 1/2$ then
\begin{eqnarray}\label{B.upper}
\overline B(x) &\le& 2\overline F(x-h_0)\quad\mbox{for all }x\in\R.
\end{eqnarray}
Also we deduce that
\begin{eqnarray*}
\overline B(x) &\le& \frac{\overline F(x-h(x))}{G(-h(x))}
\sim \overline F(x)\quad\mbox{as }x\to\infty.
\end{eqnarray*}
So, it remains to prove that
\begin{eqnarray}\label{B.lower}
\liminf_{x\to\infty}\frac{\overline B(x)}{\overline F(x)} &\ge& 1.
\end{eqnarray}
Suppose it does not hold. Then there exist an $\varepsilon>0$ 
and a sequence $x_n\to\infty$ such that
\begin{eqnarray}\label{assum.inf}
\overline B(x_n-h_0) &\le& (1-\varepsilon)\overline F(x_n-h_0)
\quad\mbox{for all }n\ge 1.
\end{eqnarray}
We have
\begin{eqnarray*}
\overline F(x_n) &=& \int_{-\infty}^{h_0} \overline B(x_n-y)G(dy)
+\int_{h_0}^\infty \overline B(x_n-y)G(dy)\\
&\le& \overline B(x_n-h_0)+\int_{h_0}^\infty \overline B(x_n-y)G(dy)\\
&\le& (1-\varepsilon)\overline F(x_n-h_0)
+2\int_{h_0}^\infty \overline F(x_n-h_0-y)G(dy),
\end{eqnarray*}
by \eqref{assum.inf} and \eqref{B.upper}. Since the distribution
$F$ is assumed to be long-tailed, the calculations of part (i) show that
\begin{eqnarray*}
\int_{h_0}^\infty \overline F(x_n-h_0-y)G(dy)
&\sim& \overline F(x_n-h_0)\overline G(h_0)\quad\mbox{as }n\to\infty.
\end{eqnarray*}
Therefore, for every $h_0$ satisfying $\overline G(-h_0)\ge 1/2$,
\begin{eqnarray*}
1=\lim_{n\to\infty}\frac{\overline F(x_n)}{\overline F(x_n-h_0)} 
&\le& 1-\varepsilon+2\overline G(h_0).
\end{eqnarray*}
Letting $h_0\to\infty$ leads to the contradiction $1\le 1-\varepsilon$. 
This justifies \eqref{B.lower} and the proof is complete.
\end{proof}

Given $X_1$ has infinitely divisible distribution, 
recall the L\'evy--Khintchine formula for the characteristic exponent 
$\Psi(\theta):=\log\E e^{i\theta X_1}$, for every $\theta\in\R$,
\begin{eqnarray*}
\Psi(\theta) &=& \Bigl(i\alpha\theta-\frac{1}{2}\sigma^2\theta^2\Bigr)
+\int_{0<|x|<1}(e^{i\theta x}-1-i\theta x)\Pi(dx)
+\int_{|x|\ge 1}(e^{i\theta x}-1)\Pi(dx)\\
&=:& \Psi_1(\theta)+\Psi_2(\theta)+\Psi_3(\theta);
\end{eqnarray*}
see, e.g. Kyprianou (2006, Sect. 2.1). 
Here $\Pi$ is the L\'evy measure concentrated on $\R\setminus\{0\}$
and satisfying $\int_\R(1\wedge x^2)\Pi(dx)<\infty$.
Let $X^{(1)}_t$, $X^{(2)}_t$ and $X^{(3)}_t$ be independent
processes given in the L\'evy--It\^o decomposition
$X_t\stackrel{d}{=} X^{(1)}_t+X^{(2)}_t+X^{(3)}_t$, where $X^{(1)}_t$ is 
a linear Brownian motion with characteristic exponent given by $\Psi^{(1)}$, 
$X^{(2)}_t$ is a square integrable martingale 
with an almost surely countable number of jumps on each finite time 
interval which are of magnitude less than unity and with
characteristic exponent given by $\Psi^{(2)}$
and $X^{(3)}_t$ is a compound Poisson process with intensity 
$\Pi(\R\setminus(-1,1))$ and jump distribution 
$\frac{\Pi(dx)}{\Pi(\R\setminus(-1,1))}$ concentrated on
$(-\infty,-1)\cup(1,\infty)$. It is known---see, 
e.g. Kyprianou (2006, Theorem 3.6) or Sato (1999, Theorem 25.17)---that the sum 
$Z_t:=X^{(1)}_t+X^{(2)}_t$ possesses all exponential moments finite,
\begin{eqnarray}\label{X12.exp}
\E e^{sZ_t}=\E e^{s(X^{(1)}_t+X^{(2)}_t)} &<& \infty 
\quad\mbox{for all }s\in\R.
\end{eqnarray}
In particular, exponential moments of $X^{(2)}_t$ may be bounded as follows. 
By the condition $\int_{(-1,1)}x^2\Pi(dx)<\infty$ 
we may produce the following upper bound:
\begin{eqnarray*}
\int_{(-1,1)}(e^{sx}-1-sx)\Pi(dx) &=& 
\int_{(-1,1)}\sum_{k=2}^\infty\frac{(sx)^k}{k!}\Pi(dx)\\
&\le& \sum_{k=2}^\infty\frac{s^k}{k!}\int_{(-1,1)}x^2\Pi(dx)\\
&=& c(e^s-1-s), 
\end{eqnarray*}
where $c:=\int_{(-1,1)}x^2\Pi(dx)$. Therefore,
\begin{eqnarray}\label{X12.exp.bound}
\E e^{sX^{(2)}_t} &=& e^{t\int_{(-1,1)}(e^{sx}-1-sx)\Pi(dx)} \le e^{cte^s}.
\end{eqnarray}

The property \eqref{X12.exp} allows to prove the following corollary 
from Lemma \ref{long.tailed.cri}.

\begin{Corollary}\label{cor:X.via.Pi}
{\rm(i)} The distribution of $X_1$ is long-tailed if and only if
the distribution of $X^{(3)}_1$ is so. In both cases,
$\P\{X_1>x\}\sim\P\{X^{(3)}_1>x\}$ as $x\to\infty$.

{\rm(ii)} The distribution of $X_1^+$ is strong subexponential if and only if
the distribution $\frac{\Pi(dx)}{\Pi(1,\infty)}$ 
concentrated on $(1,\infty)$ is so. In both cases,
$\P\{X_1>x\}\sim\Pi(x,\infty)$ as $x\to\infty$.
\end{Corollary}

\begin{proof}
The assertion (i) is immediate from Lemma \ref{long.tailed.cri}.

(ii) If $X_1^+$ has strong subexponential distribution, then it is particularly
long-tailed, so that $\P\{X_1>x\}\sim\P\{X^{(3)}_1>x\}$ as $x\to\infty$.
Hence, the distribution of $X^{(3)+}_1$ is strong subexponential too,
and in particular subexponential. Since $X^{(3)+}_1$ has compound 
Poisson distribution with parameter $\Pi(1,\infty)$ and jump distribution 
$\frac{\Pi(dx)}{\Pi(1,\infty)}$ concentrated on $(1,\infty)$,
Theorem 3 of Foss et al. (2013) yields that $\P\{X_1>x\}\sim\Pi(x,\infty)$ as $x\to\infty$.
Therefore, the distribution $\frac{\Pi(dx)}{\Pi(1,\infty)}$ concentrated on 
$(1,\infty)$ is strong subexponential---see, 
e.g. Foss et al. (2013, Corollary 3.26).

If the distribution $\frac{\Pi(dx)}{\Pi(1,\infty)}$ concentrated on 
$(1,\infty)$ is strong subexponential, then 
$\P\{X_1>x\}\sim\P\{X^{(3)+}_1>x\}\sim\Pi(x,\infty)$ by the theorem 
on tail behavior for random sums---see e.g. Foss et al. (2013, Theorem 3.37).
\end{proof}

\begin{proof}[Proof of Theorem \ref{thm:Levy.gen}.]
We start with a lower bound. We have $a=\E X^{(3)}_1+\E Z_1$. 
Fix $\varepsilon>0$ and consider two independent processes
$$
X^\varepsilon_t:=X^{(3)}_t+t\E Z_1-t\varepsilon
\quad\mbox{ and }\quad
Z^\varepsilon_t:=Z_t-t\E Z_1+t\varepsilon,
$$
so that $X_t=X^\varepsilon_t+Z^\varepsilon_t$. Then
\begin{eqnarray*}
\max_{u\in[0,t]}X_u &\ge& \max_{u\in[0,t]}X^\varepsilon_u
+\inf_{u\ge 0}Z^\varepsilon_u.
\end{eqnarray*}
Therefore, for any $x$ and $y>0$,
\begin{eqnarray*}
\P\bigl\{\max_{u\in[0,t]}X_u>x\bigr\} &\ge& 
\P\bigl\{\max_{u\in[0,t]}X^\varepsilon_u>x+y\bigr\}
\P\bigl\{\inf_{u\ge 0}Z^\varepsilon_u>-y\bigr\}.
\end{eqnarray*}
The process $Z^\varepsilon_t$ is positively driven, 
because $\E Z^\varepsilon_t=t\varepsilon>0$.
This yields that the overall minimum of the process 
$Z^\varepsilon_t$ is finite with probability $1$. 
In particular, there exists an $y_0>0$ such that
\begin{eqnarray*}
\P\bigl\{\inf_{u\ge 0}Z^\varepsilon_u>-y_0\bigr\} &\ge& 1-\varepsilon,
\end{eqnarray*}
which implies, for all $t>0$,
\begin{eqnarray}\label{deco.eps}
\P\bigl\{\max_{u\in[0,t]}X_u>x\bigr\} &\ge& (1-\varepsilon)
\P\bigl\{\max_{u\in[0,t]}X^\varepsilon_u>x+y_0\bigr\}.
\end{eqnarray}
Since $X_1^+$ is assumed to be strong subexponential,
by Corollary \ref{cor:X.via.Pi} the distribution 
$\frac{\Pi(dx)}{\Pi(1,\infty)}$ concentrated on $(1,\infty)$ is
strong subexponential too and
\begin{eqnarray*}
\overline\Pi(x) &\sim& \P\{X_1>x\}=\overline F(x)\quad\mbox{as }x\to\infty. 
\end{eqnarray*}
Then the compound Poisson process $X^\varepsilon_t$ 
with drift $(a-\varepsilon)t$ satisfies all the conditions of 
Theorem \ref{thm:renewal.linear} with $\tau$'s 
exponentially distributed which implies
\begin{eqnarray*}
\P\bigl\{\max_{u\in[0,t]}X^\varepsilon_u>x\bigr\}
&\sim& \frac{1}{|a-\varepsilon|} \int_x^{x+t|a-\varepsilon|}\overline F(v)dv
\end{eqnarray*}
as $x\to\infty$ uniformly for all $t>0$. Taking into account that
\begin{eqnarray*}
\int_x^{x+t|a-\varepsilon|}\overline F(v)dv
&\ge& \int_x^{x+t|a|}\overline F(v)dv
\end{eqnarray*}
and letting $\varepsilon\downarrow 0$,
we conclude from \eqref{deco.eps} the lower bound
\begin{eqnarray}\label{lower.Levy}
\P\bigl\{\max_{u\in[0,t]}X_u>x\bigr\}
&\ge& \frac{1+o(1)}{|a|} \int_x^{x+t|a|}\overline F(v)dv
\quad\mbox{as }x\to\infty.
\end{eqnarray}

Now proceed to prove an upper bound. Consider two independent processes
$$
X^\varepsilon_t:=X^{(3)}_t+t\E Z_1+t\varepsilon
\quad\mbox{ and }\quad
Z^\varepsilon_t:=Z_t-t\E Z_1-t\varepsilon,
$$
so that $X_t=X^\varepsilon_t+Z^\varepsilon_t$. Then
\begin{eqnarray}\label{X.Z.eps}
\max_{u\in[0,t]}X_u &\le& \max_{u\in[0,t]}X^\varepsilon_u
+\max_{u\in[0,t]}Z^\varepsilon_u.
\end{eqnarray}
Here the process $Z^\varepsilon_t$ is negatively driven, 
$\E Z^\varepsilon_t=-t\varepsilon<0$.
This yields that the overall supremum of the process 
$Z^\varepsilon_t$ is finite with probability $1$. 
Since all positive exponential moments of $Z^\varepsilon_1$ 
are finite, there exists a $\beta=\beta(\varepsilon)>0$ such that
$\E e^{\beta Z^\varepsilon_1}=1$. Then, in particular,
the Cram\'er estimate says that (see also Bertoin and Doney (1994))
\begin{eqnarray}\label{Z.eps.Cra}
\P\bigl\{\sup_{u\ge 0}Z^\varepsilon_u>x\bigr\} &\le& e^{-\beta x}.
\end{eqnarray}
We also need more accurate upper bound for 
$\P\bigl\{\sup_{u\in[0,t]}Z^\varepsilon_u>x\bigr\}$ for small values of $t$.
Notice that, for all $s>0$, the process $e^{s(Z_t-\E Z_t)}$ is a
positive submartingale, so Doob's inequality is applicable
\begin{eqnarray*}
\P\bigl\{\sup_{u\in[0,t]}Z^\varepsilon_u>x\bigr\} &\le& 
\P\bigl\{\sup_{u\in[0,t]}(Z_u-\E Z_u)>x\bigr\}\\
&\le& e^{-sx}\E e^{s(Z_t-\E Z_t)}\\ 
&=& e^{-sx}e^{s^2t\sigma^2/2}\E e^{sX^{(2)}_t}.
\end{eqnarray*}
Recalling the upper bound \eqref{X12.exp.bound} 
for $\E e^{sX^{(2)}_t}$, we get
\begin{eqnarray*}
\P\bigl\{\sup_{u\in[0,t]}Z^\varepsilon_u>x\bigr\} &\le& 
e^{-sx}e^{(s^2\sigma^2/2+ce^s)t}.
\end{eqnarray*}
For $t\le 1$, take $s:=\log\frac{1}{t}$, then 
\begin{eqnarray*}
\P\bigl\{\sup_{u\in[0,t]}Z^\varepsilon_u>x\bigr\} &\le& c_1e^{-sx}=c_1t^x.
\end{eqnarray*}
If $t\le e^{-1}$, then we finally deduce
\begin{eqnarray}\label{Z.eps.e}
\P\bigl\{\sup_{u\in[0,t]}Z^\varepsilon_u>x\bigr\} &\le& 
c_1t t^{x-1} \le ct e^{1-x}=c_2te^{-x}.
\end{eqnarray}

Since $X_1$ is assumed to be strong subexponential,
by Corollary \ref{cor:X.via.Pi} the distribution 
$\frac{\Pi(dx)}{\Pi(1,\infty)}$ concentrated on $(1,\infty)$ is
strong subexponential too. Then the compound Poisson process 
$X^\varepsilon_t$ with drift $(a+\varepsilon)t$ satisfies all the conditions 
of Theorem \ref{thm:renewal.linear} with $\tau$'s 
exponentially distributed and we have the following asymptotics
\begin{eqnarray}\label{asyX+}
\P\bigl\{\max_{u\in[0,t]}X^\varepsilon_u>x\bigr\}
&\sim& \frac{1}{|a+\varepsilon|} \int_x^{x+t|a+\varepsilon|}\overline F(v)dv\nonumber\\
&\le& \frac{1}{|a+\varepsilon|} \int_x^{x+t|a|}\overline F(v)dv
\end{eqnarray}
as $x\to\infty$ uniformly for all $t>0$. 
As follows from \eqref{Z.eps.Cra} and \eqref{Z.eps.e}, uniformly for all $t>0$,
\begin{eqnarray}\label{Z=oX}
\P\bigl\{\sup_{u\in[0,t]}Z^\varepsilon_u>x\bigr\}
&=& o\Bigl(\P\bigl\{\max_{u\in[0,t]}X^\varepsilon_u>x\bigr\}\Bigr)
\quad\mbox{as }x\to\infty.
\end{eqnarray}

Take any function $h(x)\to\infty$ such that 
$\overline F(x-h(x))\sim\overline F(x)$ as $x\to\infty$
and consider the following upper bound
\begin{eqnarray}\label{P1-P3}
\P\bigl\{\max_{u\in[0,t]}X_u>x\bigr\} &\le& 
\P\bigl\{\max_{u\in[0,t]}X^\varepsilon_u>x-h(x)\bigr\}
+\P\bigl\{\max_{u\in[0,t]}Z^\varepsilon_u>x-h(x)\bigr\}\nonumber\\
&& + \P\bigl\{\max_{u\in[0,t]}X^\varepsilon_u+\sup_{u\in[0,t]}Z^\varepsilon_u>x,\
h(x)\le \sup_{u\in[0,t]}Z^\varepsilon_u\le x-h(x)\bigr\}\nonumber\\
&:=& P_1+P_2+P_3.
\end{eqnarray}
Here the first probability $P_1$ on the right may be estimated as follows: 
by \eqref{asyX+},
\begin{eqnarray}\label{P1}
P_1 &\le& \frac{1+o(1)}{|a+\varepsilon|} \int_0^{t|a|}\overline F(x-h(x)+v)dv\nonumber\\
&\sim& \frac{1}{|a+\varepsilon|} \int_0^{t|a|}\overline F(x+v)dv  
\end{eqnarray}
By \eqref{Z=oX} and \eqref{P1}, 
\begin{eqnarray}\label{P2}
P_2 &=& o\Bigl(\P\bigl\{\max_{u\in[0,t]}X^\varepsilon_u>x-h(x)\bigr\}\Bigr)
=o\Bigl(\int_0^{t|a|}\overline F(x+v)dv\Bigr)
\quad\mbox{as }x\to\infty.
\end{eqnarray}
The probability $P_3$ is not greater than
\begin{eqnarray*}
\lefteqn{\int_{h(x)}^{x-h(x)}\P\bigl\{\sup_{u\in[0,t]}X^\varepsilon_u>x-y\bigl\}
\P\bigl\{\sup_{u\in[0,t]}Z^\varepsilon_u\in dy\bigr\}}\\
&&\hspace{25mm} \le \sum_{n=h(x)+1}^{x-h(x)}\P\bigl\{\sup_{u\in[0,t]}X^\varepsilon_u>x-n\bigl\}
\P\bigl\{\sup_{u\in[0,t]}Z^\varepsilon_u\in[n-1,n]\bigr\}\\
&&\hspace{60mm} \le c_1\sum_{n=h(x)+1}^{x-h(x)}\P\bigl\{\sup_{u\in[0,t]}X^\varepsilon_u>x-n\bigl\}
e^{-\beta n},
\end{eqnarray*}
due to the exponential upper bound \eqref{Z.eps.Cra} for $Z^\varepsilon_u$.
Then it follows from \eqref{asyX+} that
\begin{eqnarray*}
P_3 &\le& c_2\sum_{n=h(x)+1}^{x-h(x)} e^{-\beta n}
\int_0^{t|a|}\overline F(x-n+v)dv\\
&\le& c_3\int_0^{t|a|} dv
\int_{h(x)}^{x-h(x)} \overline F(x+v-y) e^{-\beta y} dy.
\end{eqnarray*}
Since $F$ is long-tailed, $e^{-\beta x}=o(\overline F(x))$.
Together with $F\in\mathcal S^*$ this implies that
\begin{eqnarray*}
\int_{h(x)}^{x-h(x)} \overline F(x+v-y) e^{-\beta y} dy
&=& o(\overline F(x+v))\quad\mbox{as }x\to\infty,
\end{eqnarray*}
so that
\begin{eqnarray}\label{P3}
P_3 &=& o\Bigl(\int_x^{x+t|a|} \overline F(v)dv\Bigr)\quad\mbox{as }x\to\infty.
\end{eqnarray}
Substituting \eqref{P1}--\eqref{P3} into \eqref{P1-P3} we obtain that
\begin{eqnarray*}
\P\bigl\{\max_{u\in[0,t]}X_u>x\bigr\}
&\le& \frac{1+o(1)}{|a+\varepsilon|} \int_x^{x+t|a|}\overline F(v)dv
\end{eqnarray*}
as $x\to\infty$ uniformly for all $t>0$. Letting $\varepsilon\downarrow 0$,
we conclude the desired upper bound
\begin{eqnarray}\label{upper.Levy}
\P\bigl\{\max_{u\in[0,t]}X_u>x\bigr\}
&\le& \frac{1+o(1)}{|a|} \int_x^{x+t|a|}\overline F(v)dv
\quad\mbox{as }x\to\infty.
\end{eqnarray}
Together with the lower bound \eqref{lower.Levy} it implies the 
required asymptotics.
\end{proof}

Similar to Theorem \ref{single.big.jump.renewal} we conclude 
with the following principle of a single big jump for the maximum
$M_t$ of the L\'evy process $X_t$. Let $T_k$ be the time epoch
of the $k$th jump of the compound Poisson process $X_t^{(3)}$
with jump absolute values greater than $1$
arising in the decomposition of $X_t$ into three independent processes.
Let $\lambda$ be the intensity of this compound Poisson process
and $Y_k$'s be its successive jumps.
Let the events $D_k$ be defined literally in the same way
as in Theorem \ref{single.big.jump.renewal}, see \eqref{def.of.Dk}.

\begin{Theorem}\label{single.big.jump.Levy}
In conditions of Theorem \ref{thm:Levy.gen}, 
for any fixed $\varepsilon>0$,
\begin{eqnarray*}
\lim_{A\to\infty}\lim_{t,x\to\infty} \P\{\cup_{k=1}^{N_t} D_k|M_t>x\}
&\ge& \frac{|a|}{|a|+2\varepsilon/\lambda}.
\end{eqnarray*}
\end{Theorem}

\section{Sampling of L\'evy process}
\label{Levy.random.time}

The last section result allows to derive tail asymptotics for a L\'evy process 
$X_t$ stopped at random time $\tau$ and for its maxima 
$M_\tau$ within this time interval.

\begin{Theorem}\label{thm1}
Assume that a positive random variable $\tau$ is independent 
of the L\'evy process $X_t$. Let the distribution $F$ of $X_1$
be strong subexponential. If $a:=\E X_1<0$ then
\begin{eqnarray}\label{eq_lim.gen}
\P\{M_\tau>x\} &\sim& \frac{1}{|a|}\E\int_x^{x+\tau|a|}\overline F(y)dy
\ \mbox{ as } x\to\infty.
\end{eqnarray}

Assume in addition that $\E\tau<\infty$. Then

{\rm(i)} If $\E X_1<0$ then
\begin{eqnarray}\label{eq_lim}
\P\{X_\tau>x\} \sim \P\{M_\tau>x\} &\sim& \E\tau\overline F(x)
\ \mbox{ as } x\to\infty.
\end{eqnarray}

{\rm(ii)} If $\E X_1\ge0$ and if there exists $c>\E X_1$ such that
\begin{eqnarray}\label{eq1}
\P\{c\tau>x\} &=& o(\overline F(x))\ \mbox{ as }x\to\infty,
\end{eqnarray}
then asymptotics \eqref{eq_lim} again hold.
\end{Theorem}

\begin{proof}
Conditioning on $\tau$ which is independent of $X_t$, we deduce that
\begin{eqnarray*}
\P\{M_\tau>x\} &=& \int_0^\infty \P\{M_t>x\}\P\{\tau\in dt\}.
\end{eqnarray*}
Then by Theorem \ref{thm:Levy.gen}, as $x\to\infty$,
\begin{eqnarray*}
\P\{M_\tau>x\} &\sim& \frac{1}{|a|} \int_0^\infty 
\int_x^{x+t|a|}\overline F(v)dv \P\{\tau\in dt\}
\end{eqnarray*}
and the first assertion \eqref{eq_lim.gen} follows.

In our proof of (i) and (ii) we follow the proof of Theorem 1 in 
Denisov et al. (2010).
Since $X_\tau \le M_\tau$, it is sufficient to prove that
\begin{eqnarray}\label{lower.bound}
\liminf_{x\to\infty} \frac{\P\{X_\tau>x\}}{\overline F(x)}
&\ge& \int_0^\infty t\P\{\tau\in dt\}=\E\tau
\end{eqnarray}
and
\begin{eqnarray}\label{upper.bound}
\limsup_{x\to\infty} \frac{\P\{M_\tau>x\}}{\overline F(x)}
&\le& \E\tau.
\end{eqnarray}

Again conditioning on $\tau$ implies
\begin{eqnarray*}
\P\{X_\tau>x\} &=& \int_0^\infty \P\{X_t>x\}\P\{\tau\in dt\}.
\end{eqnarray*}
By the subexponentiality of $X_1$, here $\P\{X_t>x\}$ is
equivalent to $t\overline F(x)$ as $x\to\infty$,
regardless of the sign of $\E X_1$. 
Then \eqref{lower.bound} follows by Fatou's lemma.

Let us now prove \eqref{upper.bound}.
If $\E X_1<0$ then \eqref{upper.bound} follows from \eqref{eq_lim.gen}
by dominated convergence due to
\begin{eqnarray*}
\int_x^{x+|a|\tau}\overline F(v)dv &\sim& 
|a|\tau\overline F(x)\quad\mbox{as }x\to\infty
\end{eqnarray*}
and upper bound
\begin{eqnarray*}
\int_x^{x+|a|\tau}\overline F(v)dv &\le& |a|\tau\overline F(x).
\end{eqnarray*}

In the case $\E X_1\ge 0$, we start with the following upper bound: for any $N$,
\begin{eqnarray}\label{prob=sum}
\P\{M_\tau>x\} &\le& \P\{M_\tau>x, \tau\le N\}
+\P\{M_\tau>x, \tau\in(N,x/c]\}+\P\{c\tau>x\}\nonumber\\
&=:& P_1+P_2+P_3.
\end{eqnarray}
By Theorem \ref{thm:Levy.1}, $\P\{M_t>x\}\sim\P\{X_t>x\}\sim t\overline F(x)$
as $x\to\infty$, for every $t$. In addition, $M_t\le M_N$ for $t\le N$.
Thus, dominated convergence yields that, for any fixed $N$,
\begin{eqnarray*}
\P\{M_\tau>x, \tau\le N\}
= \int_0^N \P\{M_t>x\} \P\{\tau\in dt\}
\sim \E\{\tau;\tau\le N\}\overline F(x)
\quad\mbox{as }x\to\infty.
\end{eqnarray*}
Therefore, there exists an increasing function $N(x)\to\infty$ such that
\begin{eqnarray}\label{first.term}
P_1=\P\{M_\tau>x, \tau\le N(x)\} \sim \E\tau\overline F(x).
\end{eqnarray}

In what follows, we consider the representation
(\ref{prob=sum}) with $N(x)$ in place of $N$.
In order to estimate $P_2$ in \eqref{prob=sum} we
take $\varepsilon=(c-\E X_1)/2>0$ and $b=(\E X_1+c)/2$.
Consider $\widetilde X_t :=X_t-bt$ and
$\widetilde M_t=\sup_{u\le t}\widetilde X_u$.
Then $\E\widetilde X_1=-\varepsilon<0$
and Theorem~\ref{thm:Levy.gen} is applicable.
Taking into account that $M_t\le\widetilde M_t+bt$,
we obtain that there exists $K$ such that,
for all $x$ and $t$,
\begin{eqnarray*}
\P\{M_t>x\} &\le& \P\{\widetilde M_t>x-bt\}\\
&\le& K\int_0^{\varepsilon t}\overline{\widetilde F}(x-bt+y)dy\\
&\le& K\int_0^{\varepsilon t}\overline F(x-bt+y)dy.
\end{eqnarray*}
Hence,
\begin{eqnarray*}
P_2=\P\{M_\tau>x, \tau\in(N(x),x/c]\}
&\le& K\int_{N(x)}^{x/c} \P\{\tau\in dt\} \int_0^{\varepsilon t}
\overline F(x-bt+y)dy.
\end{eqnarray*}
Since $b-\varepsilon=\E X_1$,
\begin{eqnarray*}
\int_0^{\varepsilon t}\overline F(x-bt+y)dy
&=& \int_{\E X_1 t}^{bt}\overline F(x-y)dy.
\end{eqnarray*}
Then
\begin{eqnarray}\label{P_2}
P_2 &\le& K\int_{N(x)\E X_1}^{bx/c} \overline F(x-y)dy
\int_{\max(N(x),y/b)}^{x/c} \P\{\tau\in dt\}
\nonumber\\
&\le& K\int_{N(x)\E X_1}^{bx/c} \overline F(x-y)
\P\{\tau> y/b\}dy.
\end{eqnarray}
Owing $b<c$ and the condition (\ref{eq1}),
the inequality $\P\{\tau>y/b\} \le K_1\overline F(y)$ 
holds for some $K_{1}$ and all $y$.
Therefore, 
\begin{eqnarray}\label{second.term}
P_2 \le KK_1\int_{N(x)\E\xi}^{bx/c}
\overline F(x-y)\overline F(y)dy
= o(\overline F(x))
\ \mbox{ as }x\to\infty
\end{eqnarray}
follows from $b/c<1$ and from $F\in{\mathcal S}^*$,
see, e.g. Foss et al. (2013, Theorem 3.24).

Finally, by the condition \eqref{eq1},
\begin{eqnarray}\label{third.term}
P_3 &=& \P\{c\tau>x\} = o(\overline F(x))
\ \mbox{ as }x\to\infty.
\end{eqnarray}
Substituting \eqref{first.term}, \eqref{second.term}, 
and \eqref{third.term} into \eqref{prob=sum} 
we conclude \eqref{upper.bound} and the proof is complete.
\end{proof}

\section{Application to ruin probabilities}
\label{sec:risk}

The results obtained above are directly applicable to 
the {\it Cram\'er--Lundberg renewal model} 
in the collective theory of risk defined as follows
(see e.g. Asmussen and Albrecher (2010, Sec. X.3)). 
We consider an insurance company and assume the constant inflow 
of premium occurs at rate $c$, that is, the premium income is assumed
to be linear in time with rate $c$. Also assume that the claims
incurred by the insurance company arrive according to
a renewal process $N_t$ with intensity $\lambda$
and the sizes (amounts) $Y_n\ge 0$ of the claims
are independent identically distributed random variables
with common distribution $B$ and mean $b$.
The $Y$'s are assumed to be independent of the process $N_t$.
The company has an initial risk reserve $u=R_0\ge0$.

Then the risk reserve $R_t$ at time $t$ is equal to
\begin{eqnarray*}
R_t &=& u+ct-\sum_{i=1}^{N_t}Y_i.
\end{eqnarray*}
Then the probability
\begin{eqnarray*}
\psi(u,t) &:=& \P\{R_s<0\mbox{ for some }s\in[0,t]\}\\
&=& \P\Bigl\{\min_{s\in[0,t]}R_s<0\Bigr\}
\end{eqnarray*}
is the finite time horizon probability of ruin. 
The techniques developed for compound renewal process with drift 
provide a method for estimating the probability of ruin 
in the presence of heavy-tailed distribution for claim sizes. We have
\begin{eqnarray*}
\psi(u,t) &=& \P\Bigl\{\sum_{i=1}^{N_s}Y_i-cs>u
\mbox{ for some }s\in[0,t]\Bigr\}.
\end{eqnarray*}
Since $c>0$, the ruin can only occur at a claim epoch. Therefore,
\begin{eqnarray*}
\psi(u,t) &=& \P\Bigl\{\sum_{i=1}^nY_i-cT_n>u
\mbox{ for some }n\le N_t\Bigr\},
\end{eqnarray*}
where $T_n$ is the $n$th claim epoch, so that
$T_n=\tau_1+\ldots+\tau_n$ where the $\tau$'s are independent
identically distributed random variables with expectation $1/\lambda$.
The last relation represents the ruin probability problem
as the tail probability problem for the maximum of 
a compound renewal process with drift.

Let the {\it net-profit condition} $c>b\lambda$ hold,
thus the process has a negative drift and $\psi(u,t)\to 0$ as $u\to\infty$,
uniformly for all $t\ge 0$. Applying Theorem \ref{thm:renewal.linear}(i),
we deduce the following result on the decreasing rate of
the ruin probability to zero as the initial risk reserve
becomes large in the case of heavy-tailed claim size distribution,
compare with a result for fixed $t$ in Section X.4 
in Asmussen and Albrecher (2010) and 
with Theorem 5.21 for the compound Poisson model in Foss et al. (2013).

\begin{Theorem}\label{thm:risk}
In the compound renewal risk model, let $c>b\lambda$.
If the claim size distribution $B$ is strong subexponential,
then, uniformly for all $t\ge0$,
\begin{eqnarray*}
\psi(u,t) &\sim& \frac{\lambda}{c-b\lambda}
\int_u^{u+t(c/\lambda-b)\E N_t}\overline B(v)dv
\quad\mbox{ as }u\to\infty.
\end{eqnarray*}
In particular,
\begin{eqnarray*}
\psi(u,t) &\sim& \frac{\lambda}{c-b\lambda}
\int_u^{u+t(c-b\lambda)}\overline B(v)dv
\quad\mbox{ as }u,\ t\to\infty.
\end{eqnarray*}
\end{Theorem}

\section*{\centering\small References}

\newcounter{bibcoun}
\begin{list}{\arabic{bibcoun}.}{\usecounter{bibcoun}\itemsep=0pt}
\small

\bibitem{APQ}
   Asmussen, S. (2003).
   {\it Applied Probability and Queues}, 2nd edn.
   Springer, New York.

\bibitem{A}
   Asmussen, S. and Albrecher, H. (2010).
   {\it Ruin Probabilities}, 2nd edn.
   World Scientific, Singapore.

\bibitem{A98}
   Asmussen, S. (1998). 
   Subexponential asymptotics for stochastic processes: 
   extremal behavior, stationary distributions and first passage probabilities. 
   {\it Annals Appl. Probab.} {\bf 2}, 354-–374.

\bibitem{AK}
   Asmussen, S. and Kl\"uppelberg, C. (1996). 
   Large deviations results for subexponential tails, 
   with applications to insurance risk. 
   {\it Stochastic Process. Appl.} {\bf 64}, 103-–125.

\bibitem{ASS}
   Asmussen, S., Schmidli, H., and Schmidt, V. (1999) 
   Tail Probabilities for non-standard risk and queueing processes 
   with subexponential jumps.
   {\it Adv. Appl. Prob.} {\bf  31}, 422--447.

%\bibitem{Ber}
%   Bertoin, J. (1996). 
%   {\it L\'evy processes}, 
%   Cambridge Univ. Press.
   
\bibitem{BD94}
   Bertoin, J. and Doney, R. (1994). 
   Cram\'er's estimate for L\'evy processes. 
   {\it Statist. Probab. Lett.} {\bf 21}, 363-–365.

\bibitem{BB}
   Borovkov, A. A. and Borovkov, K. A. (2008).
   {\it Asymptotic analysis of random walks.
   Heavy-tailed distributions},
   Cambridge Univ. Press.

\bibitem{BMS}
   Braverman, M., Mikosch, T. and Samorodnitsky, G. (2002). 
   Tail probabilities of subadditive functionals of L\'evy processes.
   {\it Ann. Appl. Probab.} {\bf 12}, 69-–100.

\bibitem{DFK} 
   Denisov, D., Foss, S., and Korshunov, D. (2010)
   Asymptotics of randomly stopped sums in the presence of heavy tails.
   {\it Bernoulli} {\bf 16}, 971--994.

\bibitem{DKM}
   Doney, R., Kl\"uppelberg, C. and Maller, R. (2016).
   Passage time and fluctuation calculations for subexponential L\'evy processes.
   {\it Bernoulli} {\bf 22}, 1491--1519.

\bibitem{EGV}
   Embrechts, P., Goldie, C., and Veraverbeke, N. (1979).
   Subexponentiality and infinite divisibility.
   {\it Z. Wahrscheinlichkeitstheorie verw. Gebiete} {\bf 49}, 335--347.

\bibitem{EKM}
   Embrechts, P., Kl\"uppelberg, C., and Mikosch, T. (1997).
   {\it Modelling Extremal Events for Insurance and Finance},
   Springer, Berlin.

\bibitem{FKonZ}
   Foss, S., Konstantopoulos, T., and Zachary, S. (2007)
   Discrete and continuous time modulated random walks 
   with heavy-tailed increments.
   {\it J. Theor. Probab.} {\bf 20}, 581-–612.

\bibitem{FKZ}
   Foss, S., Korshunov, D., and Zachary, S. (2013).
   {\it An Introduction to Heavy-Tailed and
   Subexponential Distributions}, 2nd Ed.
   Springer, New York.

\bibitem{KKM}
   Kl\"uppelberg, C., Kyprianou, A. E., and Maller, R. A. (2004).
   Ruin probabilities and overshoots for general
   L\'evy insurance risk processes.
   {\it Ann. Appl. Probab.} {\bf 14}, 1756--1801.

\bibitem{Kyprianou}
   Kyprianou, A. E. (2006).
   {\it Introductory lectures on fluctuations of L\'evy processes with applications}.
   Springer, Berlin.
   
\bibitem{Kor1997}
   Korshunov, D. (1997).
   On distribution tail of the maximum of a random walk.
   {\it Stoch. Proc. Appl.} {\bf 72}, 97--103

\bibitem{K2002}
   Korshunov, D. (2002).
   Large-deviation probabilities for maxima of sums
   of independent random variables with negative mean
   and subexponential distribution.
   {\it Theor. Probab. Appl.} {\bf 46}, 355--366

\bibitem{PZ}
   Palmowski, Z. and Zwart, B. (2007).
   Tail asymptotics of the supremum of a regenerative process.
   {\it J. Appl. Prob.} {\bf 44}, 349--365.

\bibitem{RBZ}
   Rhee, C.-H., Blanchet, J. and Zwart, B. (2016)
   Sample path large deviations for heavy-tailed L\'evy processes and random walks.
{\it arXiv preprint} arXiv:1606.02795.
            
\bibitem{RSST}
   Rolski, T., Schmidli, H., Schmidt, V, and Teugels, J.
   (1998).
   {\it Stochastic Processes for Insurance and Finance},
   Wiley, Chichester.

\bibitem{Sato}
   Sato, K. (1999).
   {\it L\'evy Processes and Infinitely Divisible Distributions},
   Cambridge Univ. Press.

\bibitem{W}
   Willekens, E. (1987).
   On the supremum of an infinitely divisible process.
   {\it Stoch. Proc. Appl.} {\bf 26}, 173--175.

\end{list}

\end{document}